\documentclass{amsart}
\usepackage{amsfonts}
\usepackage{amsmath}
\usepackage{amssymb}
\usepackage{graphicx}

\theoremstyle{plain}
\newtheorem{theorem}{Theorem}[section]
\newtheorem{lemma}[theorem]{Lemma}
\newtheorem{corollary}[theorem]{Corollary}
\newtheorem{proposition}[theorem]{Proposition}
\newtheorem{conjecture}[theorem]{Conjecture}
\theoremstyle{definition}

\newtheorem{example}[theorem]{Example}

\theoremstyle{remark}

\numberwithin{equation}{section}

\input{tcilatex}

\begin{document}
\title[Invariant ergodic measures for generalized Boole transformations]{%
Invariant measures for discrete dynamical systems and ergodic properties of
generalized Boole type transformations}
\author{Yarema A. Prykarpatsky$^{1}$}
\address{$^{1}$The Department of Applied Mathematics at the Agrarian
University, Krakow 30059, Poland\\
and\\
Department of Differential Equations of the Institute Mathematics at NAS,
Kyiv, Ukraine}
\email{yarpry@gmail.com}
\author{Denis Blackmore$^{2}$}
\address{$^{2}$Department of Mathematical Sciences and Center for Applied
Mathematics and Statistics, New Jersey Institute of Technology, Newark, NJ
07102, USA}
\email{deblac@m.njit.edu}
\author{Jolanta Golenia$^{3}$}
\address{$^{3}$The Department of Applied Mathematics at AGH University of
Science and Technology, Krakow 30059, Poland}
\email{goljols@tlen.pl}
\author{Anatoliy K. Prykarpatsky$^{4}$}
\address{$^{4}$The Department of Mining Geodesy and Environment Engineering
at AGH University of Science and Technology, Krakow 30059, Poland\\
and\\
Department of Economical Cybernetics at the Ivan Franko Pedagogical State
University, Drohobych, Lviv region, Ukraine }
\email{pryk.anat@ua.fm, pryk.anat@gmail.com}
\thanks{This work is devoted to our Friend and Teacher, the brilliant
mathematician Professor Anatoliy M. Samoilenko, on the occasion of his 75$^{%
\mathrm{th}}$ Birthday celebration.}
\date{2012}

\begin{abstract}
Invariant ergodic measures for generalized Boole type transformations are
studied using an invariant quasi-measure generating function approach based
on special solutions to the Frobenius--Perron operator. New two-dimensional
Boole type transformations are introduced, and their invariant measures and
ergodicity properties are analyzed.
\end{abstract}

\maketitle

\section{Invariant measures: introductory setting}

\noindent

It is well known that discrete dynamical systems on finite-dimensional
manifolds play an important role \cite{BMS,BPS,HPP,SUZ} in describing
evolution properties of many processes in the applied sciences. Of
particular interest are discrete dynamical systems on manifolds with
invariant measures, often possessing additional properties such as
ergodicity or mixing, which allow to explain such phenomenon as chaotic
behavior and instability of the physical objects being studied. Therefore,
methods of constructing invariant (with respect to a given discrete
dynamical system) measures, such as those we develop in the sequel, are of
crucial importance.

Suppose that a topological phase space $M$ is endowed with a structure of a
measurable space, that is a $\sigma$- algebra $\mathcal{A}(M)$ of subsets in 
$M$, on which there is a finite normalized measure $\mu:$\ $\mathcal{A}(M)$ $%
\rightarrow\mathbb{R}_{+\text{\ }},$ $\mu(M)=1$. As is well known \cite{WZ},
a measurable mapping $\varphi:M\rightarrow M$ of the measurable space ($M$, $%
\mathcal{A}(M))$ is called an \emph{ergodic} discrete dynamical system if $%
\mu$ - almost everywhere ($\mu$- a.e.) there exists an $x\in M$ limit 
\begin{equation}
\lim_{n\rightarrow\infty}\frac{1}{n}\sum_{k=0}^{n-1}f(\varphi^{k}x)
\label{E1.1}
\end{equation}
for any bounded measurable function $f\in\mathcal{B}(M;\mathbb{{R})}.$

We now assume that the limit (\ref{E1.1}) exists $\mu $- a.e., that is one
can define a bounded measurable function $f_{\varphi }\in \mathcal{B}(M;%
\mathbb{{R})},$where 
\begin{equation}
\lim_{n\rightarrow \infty }\frac{1}{n}\sum_{k=0}^{n-1}f(\varphi
^{k}x):=f_{\varphi }(x)  \label{E1.2}
\end{equation}%
for all $x\in M$ the function (\ref{E1.2}) defines a finite measure $\mu
_{\varphi }:\mathcal{A}(M)\rightarrow \mathbb{{R}_{+\text{\ }}}$ on $M$ such
that 
\begin{equation}
\int_{M}f_{\varphi }(x)\text{ }d\mu (x):=\int_{M}f(x)\text{ }d\mu _{\varphi
}(x\text{).}  \label{E1.3}
\end{equation}%
Actually, the Lebesgue--Helley theorem on bounded convergence \cite{Si}
implies that 
\begin{equation}
\int_{M}f_{\varphi }(x)\text{ }d\mu (x)=\lim_{n\rightarrow \infty
}\int_{M}f(x)\text{ }d\mu _{n,\varphi }(x),  \label{E1.4}
\end{equation}%
where 
\begin{equation}
\mu _{n,\varphi }(A):=\frac{1}{n}\sum_{k=0}^{n-1}\mu (\varphi ^{-k}A),
\label{E1.5}
\end{equation}%
is the Schur average, $n\in \mathbb{{Z}_{+}},$ $A\subset \mathcal{A}(M),$
and $\varphi ^{-k}A:=\{x\in M:\varphi ^{k}x\in A\}$ for any $k\in \mathbb{{Z}%
_{+}}$. The limit on the right hand side of (\ref{E1.4}) obviously exists
for any bounded measurable function $f\in \mathcal{B}(M;\mathbb{{R})}$.
Consequently, the equality 
\begin{equation}
\mu _{\varphi }(A):=\lim_{n\rightarrow \infty }\mu _{n,\varphi }(A),
\label{E1.6}
\end{equation}%
for any $A\subset \mathcal{A}(M)$ defines on the $\sigma $ - algebra $%
\mathcal{A}(M)$ an additive non-negative bounded mapping $\mu _{\varphi }:%
\mathcal{A}(M)$ $\rightarrow \mathbb{R}_{+}.$ Besides, from the existence of
a uniform approximation of arbitrary measurable bounded function by means of
finite-valued measurable (simple) functions, one immediately infers the
equality (\ref{E1.3}) for any $f\in \mathcal{B}(M;\mathbb{{R})}.$ The
requirement for countable additivity of the mapping $\mu _{\varphi }:%
\mathcal{A}(M)$ $\rightarrow \mathbb{{R}_{+\text{\ }}}$ follows from the
equivalent expression \cite{Sk} 
\begin{equation}
\lim_{k\rightarrow \infty }\sup_{n\in \mathbb{{Z}_{+}}}\mu _{n,\varphi
}(A_{k})=0  \label{E1.7}
\end{equation}%
for any monotonic sequence of sets $A_{j}\supset A_{j+1}$, $j\in \mathbb{{Z}%
_{+}}$, of $\mathcal{A}(M)$ with empty intersection

The measure $\mu_{\varphi}:\mathcal{A}(M)$ $\rightarrow\mathbb{{R}_{+}}$
defined by (\ref{E1.6}), has the following invariance property with respect
to the dynamical system $\varphi:M\rightarrow M$: 
\begin{equation}
\mu_{\varphi}(\varphi^{-1}A)=\mu_{\varphi}(A)  \label{E1.8}
\end{equation}
for any $\ A\in\mathcal{A}(M),$ which follows from simple identity

\begin{equation}
\mu_{n,\varphi}(\varphi^{-1}A)=\frac{n+1}{n}\mu_{n+1,\varphi}(A)-\frac{1}{n}%
\mu(A)\text{,}  \label{E1.9}
\end{equation}
upon taking the limit as $n\rightarrow\infty.$ It is easy to see that (\ref%
{E1.8}) is completely equivalent to the equality

\begin{equation}
\int_{M}f(\varphi x)d\mu_{\varphi}(x)\ =\int_{M}f(x)d\mu_{\varphi }(x)\ \ \
\ \   \label{E1.10}
\end{equation}
for any $f\in\mathcal{B}(M;\mathbb{{R})}.$ Moreover, if a $\sigma$-
measurable set $A\in\mathcal{A}(M)$ is invariant with respect to the mapping 
$\varphi:M\rightarrow M,$ that is $\varphi^{-1}(M)=M$ , then evidently $%
\mu_{\varphi}(A)=\mu(A).$\ 

Therefore, the existence of the $\varphi$-invariant measure $\mu_{\varphi }:%
\mathcal{A}(M)$ $\rightarrow\mathbb{{R}_{+}},$ coinciding with the measure $%
\mu:\mathcal{A}(M)\rightarrow\mathbb{{R}_{+}}$ on the $\sigma$-algebra $%
\mathcal{I}(M)$ $\subset\mathcal{A}(M)$ of invariant (with respect to the
dynamical system $\varphi:M\rightarrow M$ ) sets, is a necessary condition
of the convergence $\mu$- a.e. on $M$ of the mean values (\ref{E1.1}) as $%
n\rightarrow\infty$ for any $f\in\mathcal{B}(M;\mathbb{{R})}.$ That the
converse is also true follows from a theorem of Birkhoff \cite{Si}: if the
mapping $\varphi:M\rightarrow M$ conserves a finite measure $\mu_{\varphi }:%
\mathcal{A}(M)$ $\rightarrow\mathbb{{R}_{+}},$ the mean values (\ref{E1.1})
are convergent $\mu_{\varphi}$- a.e. on $M,$ and the convergence set is
invariant. Thus, if the reduction of the measure $\mu:\mathcal{A}%
(M)\rightarrow\mathbb{{R}_{+}}$ upon the invariant $\sigma-$ algebra $%
\mathcal{I}(M)$ $\subset\mathcal{A}(M)$ is absolutely continuous with
respect to that of the measure $\mu_{\varphi}:\mathcal{A}(M)$ $\rightarrow 
\mathbb{{R}_{+}}$, the convergence holds $\mu$- a.e. on $M$.

\section{An invariant measure generating construction}

Assume we are given a discrete dynamical system $\varphi :M\rightarrow M$
and a sequence of associated measures $\mu _{n,\varphi }:\mathcal{A}(M)$ $%
\rightarrow \mathbb{{R}_{+}},$ $n\in \mathbb{{Z}_{+}},$ defined by (\ref%
{E1.5}). Then one can define measure generating functions (m.g.f.) $\mu
_{n,\varphi }(\lambda ;A)$, $n\in \mathbb{{Z}_{+}}$, where for any $A\in 
\mathcal{A}(M),$ $\lambda \in \mathbb{C}$,

\begin{equation}
\ \mu_{n,\varphi}(\lambda;A):=\sum_{k=0}^{n-1}\lambda^{k}\ \mu(\varphi
^{-k}A).  \label{E2.1}
\end{equation}
Define now the following measure generating function

\begin{equation}
\mu_{\varphi}(\lambda;A):=\lim_{n\rightarrow\infty}\sum_{k=0}^{n-1}\lambda
^{k}\ \mu(\varphi^{-k}A),  \label{E2.2}
\end{equation}
where $A\in\mathcal{A}(M),$ and $\left\vert \lambda\right\vert <1$ to insure
the finiteness of the expression (\ref{E2.2}). It is easy now to prove the
following result.

\begin{lemma}
The m.g.f. (\ref{E2.2}) satisfies the functional equation 
\begin{equation}
\mu_{\varphi}(\lambda;A)=\lambda\mu_{\varphi}(\lambda;\varphi^{-1}A)+\mu(A)
\label{E2.3}
\end{equation}
for any $A\in A(M)$ and $\left\vert \lambda\right\vert <1$.
\end{lemma}

\begin{proof}
From (\ref{E2.3}) one finds by iteration directly that%
\begin{equation}
\mu _{\varphi }(\lambda ;A)-\sum_{k=0}^{n-1}\lambda ^{k}\ \mu (\varphi
^{-k}A)=\lambda ^{n+1}\mu _{\varphi }(\lambda ;\varphi ^{-k-1}A)
\label{E2.4}
\end{equation}%
for any $n\in \mathbb{Z}_{+},$ $A\in A(M)$ and $\left\vert \lambda
\right\vert <1.$ Taking the limit in (\ref{E2.4}) as $n\rightarrow \infty ,$
one arrives at the determining expression (\ref{E2.2}) that completes the
proof.
\end{proof}

\begin{corollary}
\textit{Assume we are given a mapping }$\mu ^{(s)}$\textit{\ }$:=\mu -s$ $%
\mu \circ \varphi ^{-1}$ \textit{on} $\mathcal{A}$\textit{$\mathcal{(}$M$%
\mathcal{)}$}$,$ \textit{where }$\left\vert s\right\vert <1.$ \textit{Then
the following equality}
\end{corollary}

\begin{equation}
\mu _{\varphi }^{(s)}(s;A)=\mu (A)  \label{E2.5}
\end{equation}%
\textit{holds for all } $A\in $ \textit{$\mathcal{A}$}$(M),$ $\ \left\vert
s\right\vert <1.$

\begin{proof}
This follows from a straightforward substitution of the mapping $\mu ^{(s)}:%
\mathcal{A}$\textit{$\mathcal{(}$}$M$\textit{$\mathcal{)\ }$}$\rightarrow 
\mathbb{R}$ for $\left\vert s\right\vert <1$ into (\ref{E2.3}).
\end{proof}

\begin{example}
\textit{The induced functional expansion}.
\end{example}

Let $M=[0,1]\subset\mathbb{R}$ and $\varphi:M\rightarrow M$ is the
\textquotedblleft baker\textquotedblright\ transformation, that is

\begin{equation}
\varphi (x):=\left\{ 
\begin{array}{c}
2x\text{ \ if \ \ \ }x\in \lbrack 0,1/2), \\ 
2(1-x)\text{ \ if \ }x\in \lbrack 1/2,1]%
\end{array}%
\right. .  \label{E2.6}
\end{equation}%
Take now a mapping $f:M\rightarrow M,$ given as

\begin{equation}
f(x):=2x-x^{2}  \label{E2.7}
\end{equation}%
for any $x\in M\ $\ \ and construct the convolution of \ (\ref{E2.5}) with
the function \ (\ref{E2.7}) $\ $ at the parametric measure $\ \mu
(A;x):=\int_{A}d\vartheta _{x}(y),A\in \mathit{\mathcal{A}}(M),x\ \in $ $M,$
where $\vartheta _{x}:M\rightarrow \mathbb{R},$ $\ x\ \in $ $M,$ is the
standard Heaviside function with the support \textrm{supp }$\vartheta
_{x}=\{y\in M:y-x\geq 0\}.$ Then the following decomposition

\begin{equation}
f(x)=(2-4s)\sum_{n\in \mathbb{Z}_{+}}s^{n}\varphi ^{n}(x)+(4s-1)\sum_{n\in 
\mathbb{Z}_{+}}s^{n}\varphi ^{n}(x)\varphi ^{n}(x)  \label{E2.8}
\end{equation}%
holds \cite{YH} for any $x\in M.$ In the cases \ $s=1/2$ $\ $and$\ \ s=1/4,$%
\ one readily obtains for any \ $x\in M$ the decompositions

\begin{equation}
\sum_{n\in \mathbb{Z}_{+}}(1/2)^{n}\varphi ^{n}(x)\text{ }\varphi
^{n}(x)=2x-x^{2}=\sum_{n\in \mathbb{Z}_{+}}(1/4)^{n}\varphi ^{n}(x),
\label{E2.9}
\end{equation}%
which are useful for some applied set-theoretical considerations. Note here
also that a similar expansion given by 
\begin{equation}
\sum_{n\in \mathbb{Z}_{+}}1/2^{n}\varphi ^{n}(x):=\xi (x)  \label{E2.11}
\end{equation}%
for any $x\in M,$ \ yields the well-known Weierstrass function $\xi
:[0,1]\rightarrow \lbrack 0,1],$ which is continuous but nowhere
differentiable \cite{Ta} on $M=[0,1]\subset \mathbb{R}.$

\section{Representation of invariant measures}

Assume now that the limit (\ref{E1.6}) exists owing to (\ref{E1.8}) being
measure preserving on \ $\mathcal{A}(M).$ Then the following important
Tauberian type \cite{Ha} result holds.

\begin{theorem}
\textbf{\ \ }\textit{Let the measure generating function }$\mu _{\varphi }:%
\mathbb{C}\times $\ $\mathcal{A}(M)$ $\rightarrow \mathbb{C},$ \textit{%
corresponding to a discrete dynamical system }$\varphi :M\rightarrow M,$ 
\textit{exist and satisfy the invariance condition (\ref{E1.8}). Then the
limit expression} 
\begin{equation}
\lim_{%
\begin{array}{cc}
\left. \lambda \right\uparrow 1 & (\mathrm{Im}\lambda =0)%
\end{array}%
}\mathit{\ }(1-\lambda )\mu _{\varphi }(\lambda ;A)=\mu _{\varphi }(A)
\label{E3.1}
\end{equation}%
\textit{holds for any} $A\in \mathcal{A}(M).$Moreovere, the converse is also
true.
\end{theorem}

\begin{proof}
Since all coefficients of the series (\ref{E2.1}) are bounded, that is are
of $O(1),$ then it follows from a well-known Tauberian theorem of \cite{Ha}
Hardy that 
\begin{equation}
\lim_{%
\begin{array}{cc}
\left. \lambda \right\uparrow 1 & (\text{Im}\lambda =0)%
\end{array}%
}\mathit{\ }(1-\lambda )\mu _{\varphi }(\lambda ;A)=\lim_{n\rightarrow
\infty }\frac{1}{n}\sum_{k=0}^{n-1}\mu (\varphi ^{-k}A):=\mu _{\varphi }(A)
\label{E3.2}
\end{equation}%
for any $A\in \mathcal{A}(M),$which completes the proof.
\end{proof}

We can now use the above theorem to produce an invariant measure $\mu
_{\varphi}:\mathcal{A}(M)\rightarrow\mathbb{R}_{+}$ on $M$ by means of the
measure generating function $\mu_{\varphi}:\mathbb{C}\times$ $\mathcal{A}(M)$
$\rightarrow\mathbb{C}$ defined by (\ref{E2.1}) for a given discrete
dynamical system \ $\varphi:M\rightarrow M$. Also, observe that the series (%
\ref{E2.1}) generates an analytic function when $\left\vert
\lambda\right\vert <1$ such \ that for any $\lambda\in(-1,1)$ and $A\in%
\mathcal{A}(M),$ 
\begin{equation}
\text{Im}\mu_{\varphi}(\lambda;A)=0.  \label{E3.3}
\end{equation}

Now, using classical analytic function theory \cite{PS,Pri}, one can readily
verify the following result.

\begin{theorem}
\textit{Let a measure generating function} $\mu_{\varphi}:\mathbb{C}\times $%
\ $\mathcal{A}(M)$ $\rightarrow\mathbb{C}$ \textit{satisfy the condition (%
\ref{E3.3}). Then the following representation holds}: 
\begin{equation}
\mu_{\varphi}(\lambda;A)=\int_{0}^{2\pi}\frac{(1-\lambda^{2})\text{ }%
d\sigma_{\varphi}(s;A)}{1-2\lambda\cos\text{ }s+\lambda^{2}}  \label{E3.4}
\end{equation}
\textit{for any } $A\in\mathcal{A}(M),$ \textit{where} $\sigma_{\varphi}(%
\circ;A):[0,2\pi]\rightarrow\mathbb{R}_{+}$ is \textit{a function of bounded
variation:} 
\begin{equation}
0\leq\sigma_{\varphi}(s;A)\leq\mu(A)  \label{E3.5}
\end{equation}
\textit{for any} $s\in\lbrack0,2\pi]$ \textit{and } $A\in\mathcal{A}(M).$
\end{theorem}

This theorem appears to be exceptionally interesting for applications since
it reduces the problem of detecting the invariant measure $\mu _{\varphi }:$$%
\mathcal{A}(M)$ $\rightarrow \mathbb{R}_{+}$ defined by (\ref{E1.6}) to a
calculation of the following complex analytical limit: 
\begin{equation}
\mu _{\varphi }(A)=\lim_{%
\begin{array}{cc}
\left. \lambda \right\uparrow 1 & (\text{Im}\lambda =0)%
\end{array}%
}\mathit{\ }\int_{0}^{2\pi }\frac{2(1-\lambda )^{2}\text{ }d\sigma _{\varphi
}(s;A)}{1-2\lambda \cos \text{ }s+\lambda ^{2}}\text{ ,}  \label{E3.6}
\end{equation}%
where $A\in \mathcal{A}(M)$ and $\sigma _{\varphi }:[0,2\pi ]$ $\times 
\mathcal{A}(M)\rightarrow \mathbb{R}_{+}$ - some Stieltjes measure on $%
[0,2\pi ],$ generated by a given \textit{a priori} dynamical system $\varphi
:M\rightarrow M$ and a measure $\mu :$ $\mathcal{A}(M)$ $\rightarrow \mathbb{%
R}_{+}.$

\begin{example}
\textit{The Gauss mapping.}
\end{example}

Consider the case of the Gauss mapping $\varphi :M\rightarrow M,$ where $%
M=[0,1]$ and for any $x\in (0,1],$ $\varphi (x):=\{1/x\},$ $\varphi (0)=0$
(here $\ $\textquotedblleft $\{\cdot \}$\textquotedblright\ means taking the
fractional part of a number $x\in \lbrack 0,1]).$ One can show by means of
simple but somewhat cumbersome calculations that it is indeed ergodic \cite%
{Si} and possesses the following invariant measure on $M:$%
\begin{equation}
\ \mu _{\varphi }(A)=\frac{1}{\ln 2}\int_{A}\frac{dx}{1+x}\text{ },
\label{E3.7}
\end{equation}%
which obviously yields the well-known Gauss measure $\mu _{\varphi }:$\ $%
\mathcal{A}(M)$ $\rightarrow \mathbb{R}_{+}$ on $M=(0,1].$ As a result, the
following limit for arbitrary $f\in L_{1}(0,1)$ obtains:

\begin{equation}
\lim_{n\rightarrow\infty}\sum_{k=0}^{n-1}f(\varphi^{n}x)\overset{a.e.}{=}%
\frac{1}{\ln2}\int_{0}^{1}\frac{f(x)\text{ }dx}{1+x}.  \label{E3.8}
\end{equation}

The analytical expression (\ref{E3.6}) obtained above for the invariant
measure $\mu_{\varphi}:$ $A(M)$ $\rightarrow\mathbb{R}_{+}$, generated by a
discrete dynamical system $\varphi:M\rightarrow M,$ should be quite useful
for concrete calculations. In particular, it follows directly from (\ref%
{E3.4}) that the Stieltjes measure $\sigma_{\varphi}(\circ;A):[0,2\pi]%
\rightarrow \lbrack0,\mu(A)],$ $A\in\mathcal{A}(M),$ generates for any $%
s\in\lbrack 0,2\pi]$ a new positive definite measure on $A\in\mathcal{A}(M)$
as 
\begin{equation}
\sigma_{\varphi}(s)\text{ }(A)=\sigma_{\varphi}(s;A),\text{ }  \label{E4.1}
\end{equation}
which can be regarded as smearing the measure $\mu:$\ $\mathcal{A}(M)$ $%
\rightarrow\mathbb{R}_{+}$ along the unit circle $\mathbb{\ S}^{1}$ in the
complex plane $\mathbb{C}.$

An important still open problem, which is closely linked with the expression
(\ref{E3.6}), is the following inverse measure evaluation question: How can
one retrieve the dynamical system $\varphi:M\rightarrow M$ which generated
the above smeared Stieltjes measure $\sigma_{\varphi}:[0,2\pi]$ $\times 
\mathcal{A}(M)\rightarrow\mathbb{R}_{+}$ via the expression (\ref{E3.4})?

\section{New generalizations of the Boole transformation and their ergodicity%
}

Still in 1873 British mathematicial G. Boole observed that the following
integral identity 
\begin{equation}
\int_{\mathbb{R}}f(x-1/x)dx=\int_{\mathbb{R}}f(x\ )dx  \label{ET0.0}
\end{equation}%
holds for any function $f\in L_{1}(\mathbb{R};\mathbb{R}).$ This means that
the standard Lebesgue measure $dx$ on the axis $\mathbb{R}\ $is invariant
with respect to the related mapping 
\begin{equation*}
T:\mathbb{R}\backslash \{0\}\ni x\rightarrow x-1/x\in \mathbb{R}.
\end{equation*}

In this section we will study invariant measures and ergodicity properties
of both the one-dimensional generalized Boole transformation

\begin{equation}
y\rightarrow \varphi (y):=\alpha y+a-\sum_{j=1}^{N}\frac{\beta _{j}}{y-b_{j}}%
\in \mathbb{R},  \label{ET1.1}
\end{equation}%
where $a$ and $b_{j}\in \mathbb{R}$ are real and $\alpha ,\beta _{j}\in 
\mathbb{R}_{+}$ are positive parameters, $1\leq j\leq N$, and naturally
generalized two-dimensional Boole type transformations%
\begin{align}
(x,y)& \rightarrow (x-1/x,y-1/y)\in \mathbb{R}^{2},  \notag \\
&  \label{ET1.1a} \\
(x,y)& \rightarrow (x-1/y,y-1/x)\in \mathbb{R}^{2},  \notag
\end{align}%
defined whenever $xy\neq 0$. They generalize the classical Boole
transformation \cite{Bo} $y\rightarrow \varphi (y):=y-1/y\in \mathbb{R},$
defined for $y\neq 0$, which was proved to be ergodic \cite{AW} with respect
to the invariant standard infinite Lebesgue measure on $\mathbb{R}.$ In the
case $\alpha =1,$ $a=0,$ the analogous ergodicity result was proved in \cite%
{Aa, Aaa, AaJ} making use of the specially devised inner function approach.
The related spectral properties were in part studied in \cite{AaJ}. In spite
of these results, the case $\alpha \neq 1$ still persists as a challenge. In
fact, the only related result \cite{AaaJ} concerns the following special
case of (\ref{ET1.1}): $y\rightarrow \varphi (y):=\alpha y-\beta /y\in $ $%
\mathbb{R}$ for $0<\alpha <1$ and arbitrary $\beta \in \mathbb{R}_{+},$
where the corresponding invariant measure appeared to be finite absolutely
continuous with respect to the Lebesgue measure on $\mathbb{R}$ and equal to 
\begin{equation}
d\mu (x):=\frac{\sqrt{\beta (1-\alpha )}dx}{\pi \lbrack x^{2}(1-\alpha
)+\beta ]},  \label{ET1.2}
\end{equation}%
where $x\in \mathbb{R}.$ The ergodicity for the invariant measure (\ref%
{ET1.2}) now can be easily proved. It should be recalled here that for a
general nonsingular mapping $\varphi :\mathbb{R}\rightarrow \mathbb{R},$ the
problem of constructing invariant ergodic measures is analyzed\ \cite{AaaJ,
KH} by studying the spectral properties of the adjoint Frobenius--Perron
operator $\hat{T_{\varphi }}\ :L_{^{2}}(\mathbb{R};\mathbb{R})\rightarrow
L_{2}(\mathbb{R};\mathbb{R}),$ where

\begin{equation}
\hat{T_{\varphi }}\rho (x):=\sum_{y\in \{\varphi ^{-1}(x)\}}\rho
(y)J_{\varphi }^{-1}(y)  \label{ET1.3}
\end{equation}%
for any $\rho \in L_{2}^{\ \ }(\mathbb{R};\mathbb{R}_{+})$ and $J_{\varphi
}^{-1}(y):=|\frac{d\varphi (y)}{dy}|,$ $y\in \mathbb{R}.$ Then if $\ {\hat{T}%
_{\varphi }}\rho =\rho ,$ $\rho \in L_{2}(\mathbb{R};\mathbb{R}_{+}),$ the
expression $d\mu (x):=\rho (x)dx,$ $x\in \mathbb{R},$ will be an invariant
(in general infinite) measure with respect to the mapping $\varphi :\mathbb{R%
}\rightarrow \mathbb{R}.$

Another way of finding a general algorithm for obtaining such an invariant
measure was devised in \cite{Pr, PB} using the generating measure function
method.

Below we study some other special cases of the generalized Boole
transformation (\ref{ET1.1}), for which we derive the corresponding
invariant measures and prove the related ergodicity and spectral properties.

\subsection{Invariant measures and ergodic transformations}

We will start with analyzing the following Boole type surjective
transformation

\begin{equation}
\mathbb{R}\ni y\rightarrow\varphi(y):=\alpha y+a-\frac{\beta}{y-b}\in\mathbb{%
R}  \label{ET2.1}
\end{equation}
for any $a,b\in\mathbb{R}$ and $2\beta:=\gamma^{2}\in\mathbb{R}_{+}.$ The
transformation (\ref{ET2.1}) for $\alpha=1/2$ and $b=2a\in\mathbb{R}$ is
measure preserving with respect to a measure like (\ref{ET1.2}). Namely, the
following lemma holds.

\begin{lemma}
The Boole type mapping (\ref{ET2.1}) is measure preserving with respect to
the measure
\end{lemma}

\begin{equation}
d\mu(x):=\frac{|\gamma|dx}{\pi\lbrack(x-2a)^{2}+\gamma^{2}]},  \label{ET2.2}
\end{equation}
\textit{where }$x\in\mathbb{R}$\textit{\ and }$\gamma^{2}:=2\beta\in \mathbb{%
R}_{+}\mathit{.}$

\begin{proof}
A proof follows easily from the fact that the function%
\begin{equation}
\rho(x):=\frac{\gamma}{\pi\lbrack(x-2a)^{2}+\gamma^{2}]}  \label{ET2.3}
\end{equation}
satisfies for all $x\in\mathbb{R}\backslash\{{2a\}}$ the determining
condition (\ref{ET1.3}):%
\begin{equation}
\hat{T_{\varphi}}\rho(x):=\sum_{I}\rho(y_{\pm})|y_{\pm}^{\prime}(x)|,
\label{ET2.4}
\end{equation}
where, $\varphi(y_{\pm}(x)):=x$ for any $x\in\mathbb{R}.$ The relationship (%
\ref{ET2.4}) is manifestly equivalent to the invariance condition%
\begin{equation}
\sum_{\pm}d\mu(y_{\pm}(x))=d\mu(x):=\mu(dx)  \label{ET2.5}
\end{equation}
for any infinitesimal subset $dx\subset\mathbb{R}$.
\end{proof}

The question about the ergodicity of the mapping (\ref{ET2.1}) is solved
here easily by the following theorem.

\begin{theorem}
The measure (\ref{ET2.3}) is ergodic with respect to the transformation (\ref%
{ET2.1}) at $\alpha =1/2$ and $b=2a\in \mathbb{R}$ as such one is equivalent
to the canonical ergodic mapping $\mathbb{R}/\mathbb{Z}\ni s:\rightarrow
\psi (s):=2s$ \textrm{mod }$\mathbb{Z}\in \mathbb{R}/\mathbb{Z}$ with
respect to the standard Lebesgue measure on $\mathbb{R}/\mathbb{Z}.$
\end{theorem}

\begin{proof}
Define $\mathbb{R}/\mathbb{Z}\ni s:\rightarrow\xi(s)=y\in\mathbb{R},$ where%
\begin{equation}
\xi(s):=\gamma\cot\pi s+2a,  \label{ET2.6}
\end{equation}
Then the transformation (\ref{ET2.1}) for $\alpha=1/2,$ $b=2a\in\mathbb{R}$
and $\gamma^{2}:=2\beta\in\mathbb{R}_{+}$ yields under the mapping (\ref%
{ET2.6})%
\begin{align}
\varphi(y) & =\varphi(\xi(s))=\frac{\gamma}{2}\cot\pi s+2a-\frac{\gamma}{2}%
\tan\pi s=\frac{\gamma(\cos^{2}\pi s-\sin^{2}\pi s)}{2\sin\pi s\cos\pi s}+2a
\notag \\
& =\gamma\frac{\cos2\pi s}{\sin2\pi s}+2a=\gamma\cot2\pi s+2a:=\xi (2s)
\label{ET2.7}
\end{align}
for any $s\in\mathbb{R}/\mathbb{Z}.$ The result (\ref{ET2.6}) means that the
transformation (\ref{ET2.1}) is conjugated \cite{KH,AaJ} with the
transformation%
\begin{equation}
\mathbb{R}/\mathbb{Z}\ni s:\rightarrow\psi(s)=2s\text{ \textrm{mod }}\mathbb{%
Z}\in\mathbb{R}/\mathbb{Z};  \label{ET2.8}
\end{equation}
that is, the following diagram is commutative:%
\begin{equation}
\begin{array}{ccc}
\mathbb{R}/\mathbb{Z} & \overset{\psi}{\rightarrow} & \mathbb{R}/\mathbb{Z}
\\ 
\xi\downarrow &  & \downarrow\xi \\ 
\mathbb{R} & \overset{\varphi}{\rightarrow} & \mathbb{R},%
\end{array}
\label{ET2.9}
\end{equation}
that is $\xi\circ\psi=\varphi\circ\xi,$ where $\xi:\mathbb{R}/\mathbb{%
Z\rightarrow R}$ is the conjugate map defined by (\ref{ET2.6}). It is easy
now to check that the measure (\ref{ET2.2}) under the conjugation (\ref%
{ET2.9}) transforms into the standard normalized Lebesgue measure on $%
\mathbb{R}/\mathbb{Z}:$%
\begin{equation}
d\mu(x)|_{x=\gamma\cot\pi s+2a}=\frac{ds\gamma^{2}\left\vert d(\cot\pi
s)/ds\right\vert }{(\gamma^{2}\cot^{2}\pi s+\gamma^{2})}  \label{ET2.10}
\end{equation}%
\begin{equation*}
=\frac{\sin^{2}\pi s\cdot(\sin\pi s)^{-2}\;ds}{\cos^{2}\pi s+\sin^{2}\pi s}%
=ds,
\end{equation*}
where $s\in\mathbb{R}/\mathbb{Z}.$ The infinitesimal measures $ds$ on $%
\mathbb{R}/\mathbb{Z}$ and the infinitesimal measure (\ref{ET2.2}) on $%
\mathbb{R}$ are normalized, so they are both probability measures. Now it
suffices to make use of the fact that the measure $ds$ on $\mathbb{R}/%
\mathbb{Z}$ on the interval $[0,1]\simeq\mathbb{R}/\mathbb{Z}$ is ergodic 
\cite{AaaJ,KH} in order to obtain the desired result.
\end{proof}

\subsection{Ergodic measures: the inner function approach}

Assume that there exists a function $\rho _{\omega }\in H_{2}(\mathbb{C}_{+};%
\mathbb{C}),$ holomorphic in parameter $\omega \in \mathbb{C}_{+},$
satisfying the following identity

\begin{equation}
\hat{T_{\varphi}}\rho_{\omega}=\rho_{\tilde{\varphi}(\omega)}  \label{ET3.1}
\end{equation}
for any $\omega\in\mathbb{C}_{+}$ for some induced transformation $\mathbb{C}%
_{+}\ni\omega\rightarrow\tilde{\varphi}(\omega)\in\mathbb{C}_{+}.$ If we now
take $\omega:=\bar{\omega}\in\mathbb{C}_{+}$ as a fixed point of the mapping 
$\tilde{\varphi}:\mathbb{C}_{+}\rightarrow\mathbb{C}_{+},$ then it follows
directly from \ref{ET3.1} that $\hat{T_{\varphi}}\rho_{\bar{\omega}}=\rho_{%
\bar{\omega}},$ which means

\begin{equation}
d\mu (x):=\mathrm{Im}\rho _{\bar{\omega}}(x)dx  \label{ET3.2}
\end{equation}%
for $x\in \mathbb{R}$ is an invariant measure for the transformation $%
\varphi :\mathbb{R}\rightarrow \mathbb{R}.$ There is no general rule for
constructing such functions $\rho _{\omega }\in H_{2}(\mathbb{C}_{+};\mathbb{%
C})$, analytic in $\omega \in \mathbb{C}_{+},$ and the related induced
mappings $\tilde{\varphi}:\mathbb{C}_{+}\rightarrow \mathbb{C}_{+}.$
Nevertheless, for solving this problem one can adapt some natural ideas
related to the exact functional form of the determining Frobenius--Perron
operator $\hat{T_{\varphi }}:L_{2}^{\ \ }(\mathbb{R};\mathbb{R})\rightarrow
L_{2}^{\ \ }(\mathbb{R};\mathbb{R}).$ To explain this, let us consider the
following Boole type transformation:

\begin{equation}
\mathbb{R}\ni\varphi(y):=\alpha y+a-\frac{\beta}{y-b}\in\mathbb{R},
\label{ET3.3}
\end{equation}
where $\ {a,b\in}\mathbb{R}$ and $\beta\in\mathbb{R}_{+}.$ It is easy to see
that the Frobenius-Perron operator action on any $\rho_{\omega}\in H_{2}(%
\mathbb{C}_{+};\mathbb{C})$ can be represented as follows:

\begin{equation}
\begin{split}
\hat{T_{\varphi}}\rho_{\omega}:= &
\rho_{\omega}(y_{+})y_{+}^{\prime}+\rho_{\omega}(y_{-})y_{-}^{\prime} \\
= & \frac{(\omega-y_{+})\rho_{\omega}(y_{+})(\omega-y_{-})y_{-}^{\prime}}{%
(\omega-y_{+})(\omega-y_{-})}+\frac{\rho_{\omega}(y_{-})(\omega
-y_{+})(\omega-y_{-})y_{-}^{\prime}}{(\omega-y_{+})(\omega-y_{-})} \\
= & \frac{k(\omega-y_{-})y_{+}^{\prime}+k(\omega-y_{+})y_{-}^{\prime}}{%
(\omega-y_{+})(\omega-y_{-})}=\frac{-k[(\omega-y_{+})(\omega-y_{-})]^{\prime}%
}{(\omega-y_{+})(\omega-y_{-})} \\
= & -k\frac{d}{dx}ln[(\omega-y_{+})(\omega-y_{-})],
\end{split}
\label{ET3.4}
\end{equation}
where

\begin{equation}
\rho_{\omega}(x)=\frac{k}{\omega-x}  \label{ET3.5}
\end{equation}
for all $\omega\in\mathbb{C}_{+}\backslash\{x\},$ $x\in$ $\mathbb{R},$ and
some parameter $k\in\mathbb{R}.$ As a result of (\ref{ET3.5}), one can take

\begin{equation}
\rho_{\omega}(y_{+})(\omega-y_{+})=k=\rho_{\omega}(y_{-})(\omega -y_{-}),
\label{ET3.6}
\end{equation}
for all $x\in\mathbb{R}$ and $\omega\in\mathbb{C}_{+}.$ Since the root
functions $y_{+}$ and $y_{-}$ $:\mathbb{R}\rightarrow\mathbb{R}$ satisfy, by
definition, the same equation

\begin{equation}
\varphi (y_{\pm }(x))=x,  \label{ET3.7}
\end{equation}%
for all $x\in \mathbb{R},$ the following identity for all $\omega \in 
\mathbb{C}_{+}$ easily follows from (\ref{ET3.7}) owing to the general form
of (\ref{ET3.3}):

\begin{equation}
\alpha(\omega-y_{+})(\omega-y_{-})=[\varphi(\omega)-x](\omega-b),
\label{ET3.8}
\end{equation}
where

\begin{equation}
y_{+}(x)+y_{-}(x)=b+\frac{x-a}{2},~y_{+}(x)y_{-}(x)=\frac{bx-ab-\beta}{2}.
\label{ET3.9}
\end{equation}
Whence, taking into account the expression (\ref{ET3.4}), one computes that

\begin{equation}
\begin{split}
\hat{T_{\varphi}}\rho_{\omega}= & -k\frac{d}{dx}\mathrm{ln}([\varphi
(\omega)-x](\omega-b)) \\
= & \frac{k(\omega-b)}{[\varphi(\omega)-x](\omega-b)}=\frac{k}{%
\varphi(\omega)-x}=\rho_{\varphi(\omega)},
\end{split}
\label{ET3.10}
\end{equation}
for all $x\in\mathbb{R}$ and $\omega\in\mathbb{C}_{+}.$ Therefore, the
induced mapping $\tilde{\varphi}:\mathbb{C}_{+}\rightarrow\mathbb{C}_{+}$ is
exactly the transformation $\varphi:\mathbb{C}_{+}\rightarrow\mathbb{C}_{+},$
extended naturally from $\mathbb{R}$ to the complex plane $\mathbb{C}_{+}.$

Now let $\tilde{\omega}\in\mathbb{C}_{+}$ be a fixed point of the induced
mapping $\varphi:\mathbb{C}_{+}\rightarrow\mathbb{C}_{+},$ that is $\varphi(%
\bar{\omega})=\bar{\omega}\in\mathbb{C}_{+}.$ Then from (\ref{ET3.10}), one
finds that $\hat{T_{\varphi}}\rho_{\bar{\omega}}=\rho_{\bar{\omega}},$ or
the corresponding invariant quasi-measure on $\mathbb{R}$ has the form

\begin{equation}
d\mu(x):=\mathrm{Im}\frac{kdx}{\bar{\omega}-x}  \label{ET3.11}
\end{equation}
for all $x\in\mathbb{R}$ and a suitable parameter $k\in\mathbb{C}.$ As $%
\mathrm{Im}\rho_{\bar{\omega}}\in L_{2}(\mathbb{R};\mathbb{R}_{+})$ at any $%
\bar{\omega}\in\mathbb{C}_{+}\backslash\mathbb{R}$ and some $k\in\mathbb{C},$
the invariant quasi-measure (\ref{ET3.11}) transforms into an actual
invariant measure. These results can be formulated as follows:

\begin{theorem}
\label{Th_ET3.1} The quasi-measure (\ref{ET3.11}) is invariant with respect
to the transformation (\ref{ET3.3}) for any $\alpha\in\mathbb{R}%
_{+}\backslash\{{1\};}$ for $\alpha=1$ at the condition $a\neq0,$ $\mathrm{Im%
}k\neq0,$ it is reduced upon the set $\mathbb{R}/\pi\mathbb{Z},$ being
equivalent to the standard Gauss measure. \ 
\end{theorem}

\begin{proof}
The desired infinitesimal quasi-measure $d\mu(x)$ exist if there is at least
one fixed point of the equation $\varphi(\omega)=\omega$ for $\omega \in%
\mathbb{C}_{+}.$ If $\alpha\neq1,$ this equation is equivalent to 
\begin{equation}
(\alpha-1)\omega^{2}-\omega\lbrack(\alpha-1)b-a]-(ab+\beta)=0,
\label{ET3.12}
\end{equation}
which always has a solution $\bar{\omega}\in\mathbb{C}_{+},$ for which $%
\varphi(\bar{\omega})=\bar{\omega}.$ When $\alpha=1$, the unique solution $%
\bar{\omega}=(ab+\beta)/a\in\mathbb{R}$ exists only if $a\neq0$ and $\mathrm{%
Im}k\neq0,$ at which the quasi-measure (\ref{ET3.11}) becomes degenerate and
reduces to the standard Gauss measure \cite{KH,Aa} on $\mathbb{R}/\pi\mathbb{%
Z}$.
\end{proof}

Theorem \ref{Th_ET3.1} states only that the quasi-measure (\ref{ET3.11}) is
invariant with respect to the transformation (\ref{ET3.3}), so its
ergodicity still needs to be proved separately using only the additional
property that the corresponding invariant measure is unique. Below we will
proceed to study the general case of the transformation (\ref{ET1.1}),
searching for a suitable invariant quasi-measure that is actually a measure
for some $\bar{\omega}\in \mathbb{C}_{+}\backslash \mathbb{R},$ $k\in 
\mathbb{C}.$

\subsection{Invariant measures: the general case}

Consider the following equation

\begin{equation}
\varphi (y)=x,  \label{ET4.1}
\end{equation}%
where $x,y\in \mathbb{R}$ and the mapping $\varphi :\mathbb{C}%
_{+}\rightarrow \mathbb{C}_{+}$ is given by expression (\ref{ET1.1}) for a
fixed integer $N\in \mathbb{Z}_{+}\backslash \{0,{1\}.}$ The equation (\ref%
{ET4.1}) can be rewritten as

\begin{equation}
\alpha\prod_{j=1}^{N+1}(y-y_{j})=[\varphi(y)-x]\prod_{j=1}^{N}(y-b_{j})
\label{ET4.2}
\end{equation}
for all $x,y\in\mathbb{R}$ and some functions $y_{j}:\mathbb{R}\rightarrow 
\mathbb{R},$ $1\leq j\leq N+1.$ Then the relationship (\ref{ET4.2}) is
naturally extended on the complex plane $\mathbb{C}_{+}$ as

\begin{equation}
\alpha\prod_{j=1}^{N+1}(\omega-y_{j})=[\varphi(\omega)-x]\prod_{j=1}^{N}(%
\omega-b_{j})  \label{ET4.3}
\end{equation}
for any $\omega\in\mathbb{C}_{+}.$

Consider now the relationship (\ref{ET3.1}) in the manner of Section 3;
namely

\begin{equation*}
\hat{T_{\varphi }}\rho _{\omega }=\sum_{j=1}^{N+1}\rho _{\omega
}(y_{j})y_{j}^{\prime }=\sum_{j=1}^{N+1}\frac{\rho _{\omega }(y_{j})(\omega
-y_{j})\prod_{k\neq j}^{N+1}(\omega -y_{k})y_{j}^{\prime }}{%
\prod_{k=1}^{N+1}(\omega -y_{k})}
\end{equation*}

\begin{equation}
\begin{split}
=& \sum_{j=1}^{N+1}\frac{\rho _{\omega }(y_{j})(\omega -y_{j})\prod_{k\neq
j}^{N+1}(\omega -y_{k})y_{j}^{\prime }}{\prod_{k=1}^{N+1}(\omega -y_{k})}%
=\sum_{j=1}^{N+1}\frac{k\prod_{k\neq j}^{N+1}(\omega -y_{k})y_{j}^{\prime }}{%
\prod_{k=1}^{N+1}(\omega -y_{k})} \\
& \\
=& -k\frac{\frac{d}{dx}\prod_{k}^{N+1}(\omega -y_{k})}{\prod_{k=1}^{N+1}(%
\omega -y_{k})}=-k\frac{d}{dx}\ln [\alpha ^{-1}\prod_{k=1}^{N+1}(\omega
-y_{k})\prod_{j=1}^{N}(y-b_{j})] \\
& \\
& =-k\frac{d}{dx}\ln [\varphi (\omega )-x]=\frac{k}{\varphi (\omega )-x},
\end{split}
\label{ET4.4}
\end{equation}%
where we have put, as before,

\begin{equation}
\rho_{\omega}(y_{j})(\omega-y_{j})=k,  \label{ET4.5}
\end{equation}
for all $j=1,\ldots,N+1,$ $\omega\in\mathbb{C}_{+},$ and some parameters $%
k\in\mathbb{C}.$ This clearly means that

\begin{equation}
\rho_{\omega}(y)=\frac{k}{\omega-y}  \label{ET4.6}
\end{equation}
for any $y\in\mathbb{R}$ and $\omega\in\mathbb{C}_{+}$.

Upon substituting the expression (\ref{ET4.3}) into (\ref{ET4.4}), one
readily finds that

\begin{equation}
\hat{T_{\varphi}}\rho_{\omega}(x)=\frac{k}{\varphi(\omega)-x}=\rho
_{\varphi(\omega)}(x),  \label{ET4.7}
\end{equation}
for all $x\in\mathbb{R}$ and any $\omega\in\mathbb{C}_{+}.$ Thus, the
invariant quasi-measure for the discrete dynamical system (\ref{ET1.1}) is
given by the same expression (\ref{ET3.11}) when $\bar{\omega}\in \mathbb{C}%
_{+}$ is a fixed point of the mapping $\varphi:\mathbb{C}_{+}\rightarrow%
\mathbb{C}_{+}.$ This means that

\begin{equation}
\alpha\bar{\omega} +a- \sum_{j=1}^{N} \frac{\beta_{j}}{\bar{\omega} -b_{j}}= 
\bar{\omega},  \label{ET4.8}
\end{equation}
or, equivalently,

\begin{equation}
\alpha\bar{\omega}\prod_{j=1}^{N}(\bar{\omega}-b_{j})+a\prod_{j=1}^{N}(\bar{%
\omega}-b_{j})-\sum_{j=1}^{N}\beta_{j}\prod_{k\neq j}^{N}(\bar{\omega }%
-b_{k})=\bar{\omega}\prod_{j=1}^{N}(\bar{\omega}-b_{j}),  \label{ET4.9}
\end{equation}
for some $\bar{\omega}\in\mathbb{C}_{+}.$ Assume now that $\alpha\neq1;$
then it is easy to see that the algebraic equation (\ref{ET4.9}) possesses
exactly $N+1\in\mathbb{Z}_{+}$ roots, which can be used to constructing the
invariant quasi-measure (\ref{ET3.11}). When $\alpha=1$, the condition
becomes

\begin{equation}
a\prod_{j=1}^{N}(\bar{\omega}-b_{j})=\sum_{j=1}^{N}\beta_{j}\prod_{k\neq
j}^{N}(\bar{\omega}-b_{k}),  \label{ET4.10}
\end{equation}
which always possesses roots for arbitrary $a\in\mathbb{R}$ if $N\geq2$.
This leads directly to the following characterization for $N\geq2$:

\begin{theorem}
The expression (\ref{ET3.11}) for some $k\in\mathbb{C}$ determines, in
general, the infinitesimal invariant quasi-measure for the generalized Boole
transformation (\ref{ET1.1}) for all $N\geq2$ with arbitrary parameters $a,$ 
$b_{j}\in\mathbb{R}$ and $\alpha,\beta_{j}\in\mathbb{R}_{+},$ $1\leq j\leq
N+1.$
\end{theorem}

It is an important now to find in the set of invariant quasi-measures (\ref%
{ET3.11}) that we obtained, those that are positive and ergodic with respect
to the transformation (\ref{ET1.1}) for $N\geq2.$ For positivity, the
determining equation (\ref{ET4.9}) must possess at least one pair of complex
conjugate roots. A thorough analysis of the roots of equation (\ref{ET4.9})
leads to the following result, which is analogous to that proved in \cite%
{AaaJ}.

\begin{theorem}
The generalized Boole transformation (\ref{ET1.1}) for any $N\geq1$ is
necessarily ergodic with respect to the measure (\ref{ET3.11}) for some $%
\bar{\omega}\in\mathbb{C}_{+}\backslash\mathbb{R}$ and $k\in\mathbb{C}$ iff $%
\ \alpha=1$ and $a=0.$ If $\alpha=1$ and $a\neq0,$ the transformation (\ref%
{ET1.1}) is not ergodic since it is totally dissipative, that is the
wandering set $\ \mathcal{D}(\varphi):=\bigcup\mathcal{W}_{\varphi}=\mathbb{R%
},$ where $\mathcal{W}_{\varphi}\subset\mathbb{R}$ are such subsets such
that ${\varphi^{-n}(\mathcal{W}_{\varphi}\mathcal{)}},$ ${n\in \mathbb{Z},}$
are disjoint.
\end{theorem}

\begin{proof}
(sketch). It is easy to see that for $N\geq 2,$ $\alpha =1$ and $a=0$ the
determining algebraic equation (\ref{ET4.9}) always possesses exactly $N-1$
real roots $\bar{\omega _{j}}\in \mathbb{R},$ $j=1,\ldots ,N-1.$ Therefore,
the invariant quasi-measure expression (\ref{ET3.11}) is degenerate for all
of the $\bar{\omega _{j}}\in \mathbb{R},$ which leads directly to the
conclusion that the corresponding invariant measure $d\mu (x)=dx,$ $x\in 
\mathbb{R},$ is the standard Lebesgue measure on $\mathbb{R}.$ Its
ergodicity with respect that transformation (\ref{ET1.1}) then follows from
the fact that the corresponding dissipative set $\mathcal{D}(\varphi
)=\varnothing $ and the unique invariant set subalgebra $I(\varphi
)=\{\varnothing ,{\mathbb{R\}}}.$
\end{proof}

Results similar to those above can also be obtained for the most generalized
Boole type transformation

\begin{equation}
\mathbb{R}\ni y\rightarrow\varphi(y):=\alpha y+a+\int_{\mathbb{R}}\frac {%
d\nu(s)}{s-y}\in\mathbb{R},  \label{ET4.11}
\end{equation}
where $a\in\mathbb{R},$ $\alpha\in\mathbb{R}_{+}$ and the measure $\nu$ on $%
\mathbb{R}$ has the compact support $\mathrm{supp}$ $\nu\subset\mathbb{R}$
such that the following natural conditions \cite{Aa,KN}

\begin{equation}
\int_{\mathbb{R}}\frac{d\nu(s)}{1+s^{2}}=a,~\int_{\mathbb{R}}d\nu (s)<\infty,
\label{ET4.12}
\end{equation}
hold. Concerning the extension of the transformation (\ref{ET4.11}) on the
upper part $\mathbb{C}_{+}$ of the complex plane $\mathbb{C}$ so that that $%
\mathrm{Im}\varphi(\omega)\geq0$ for all $\omega\in\mathbb{C}_{+},$ the
following representation

\begin{equation}
\varphi(\omega)=\alpha\omega+a+\int_{\mathbb{R}}\frac{1+s\omega}{s-\omega }%
d\sigma(s),  \label{ET4.13}
\end{equation}
holds \cite{AF,Aa}, where the measure $d\sigma$ on $\mathbb{R}$ is closely
related to the measure $d\nu.$

The general properties of the mapping (\ref{ET4.13}) were in part studied in 
\cite{AaaJ} in the framework of the theory of inner functions. The invariant
measures corresponding to (\ref{ET4.11}) and their ergodic properties can be
also treated effectively by making use of the analytical and spectral
properties of the associated Frobenius--Perron transfer operator (\ref{ET1.3}%
).

\section{Two-dimensional generalizations of the Boole transformation}

Consider the two-dimensional Boole type transformations $\varphi _{2},\psi
_{\sigma (2)\ }:\mathbb{R}^{2}\backslash \{0,0\}\rightarrow \mathbb{R}^{2}$ 
\begin{equation}
\varphi _{2}(x,y):=(x-1/x,y-1/y)  \label{ET5.1a}
\end{equation}%
and 
\begin{equation}
\psi _{\sigma (2)\ }(x,y):=(x-1/y,y+1/x)  \label{ET5.1b}
\end{equation}%
It is easy to see that the infinitesimal (product) measure 
\begin{equation}
d\mu (x,y):=dxdy,  \label{ET5.1c}
\end{equation}%
is invariant with respect to the first mapping (\ref{ET5.1a}) since it is
the product of two measures, each of which is invariant with respect to the
corresponding classical Boole transformation. Therefore, the generalized
Boole type transformation (\ref{ET5.1a}) is ergodic too.    \ \ \ 

We have \ observed that the infinitesimal Lebesgue measure $d\lambda
(x,y):=dxdy,$ $(x,y)\in \mathbb{R}^{2},$ on the plane $\mathbb{R}^{2}$ \ is
invariant subject to the mapping (\ref{ET5.1b}), that can be easily checked
making use of the Perron-Frobenius condition: for the corresponding
preimages $(u_{\pm },v_{\pm }):=$ $(u_{\pm }(x,y),v_{\pm }(x,y))$ $\in 
\mathbb{R}^{2},$ where $%
u_{+}u_{-}=xy^{-1},v_{+}v_{-}=-yx^{-1},u_{+}+u_{-}=2y^{-1}+x,$ $%
v_{+}+v_{-}=y-2x^{-1},$ $\varphi (u_{\pm },v_{\pm })=(x,y)\in \mathbb{R}^{2},
$ one verifies that the measure 
\begin{equation}
\begin{array}{c}
\sum_{\pm }du_{\pm }dv_{\pm }(x,y)=\sum_{\pm }\left\vert J_{(u_{\pm },v_{\pm
})}(x,y)\right\vert dxdy= \\ 
=\sum_{\pm }\frac{dxdy}{|J_{\varphi }(u_{\pm },v_{\pm })|}=\sum_{\pm }\frac{%
dxdy}{(1+\left( u_{\pm }v_{\pm }\right) ^{-2})}= \\ 
=\sum_{\pm }\frac{\left( u_{\pm }v_{\pm }\right) ^{2}dx}{(1+\left( u_{\pm
}v_{\pm }\right) ^{2})}=\frac{\left[ 2(u_{+}v_{+}u_{-}v_{-})^{2}\ +\left(
u_{-}v_{-}\right) ^{2}+\left( u_{+}v_{+}\right) ^{2}\right] dxdy}{\left[
1+\left( u_{-}v_{-}\right) ^{2}+\left( u_{+}v_{+}\right)
^{2}+(u_{+}v_{+}u_{-}v_{-})^{2}\right] }= \\ 
=\frac{\ \left[ \left( u_{-}v_{-}\right) ^{2}+\left( u_{+}v_{+}\right) ^{2}+2%
\right] dxdy}{\left[ 2+\left( u_{-}v_{-}\right) ^{2}+\left(
u_{+}v_{+}\right) ^{2}\right] }=dxdy,%
\end{array}
\label{ET5.2}
\end{equation}%
coinciding exactly with the Lebesgue measure $d\lambda (x,y):=dxdy,(x,y)\in 
\mathbb{R}^{2}.$ \  

Consequently, it follows by a simple modification of the proof of the main
theorem in \cite{AW} that $\ d\mu (x,y),(x,y)\in \mathbb{R}^{2},$ is the
unique absolutely continuous invariant measure for the Boole type
transformation (\ref{ET5.1b}). This, in particular, implies that the map (%
\ref{ET5.1b}) is ergodic with respect to the infinitesimal measure $d\mu
(x,y),(x,y)\in \mathbb{R}^{2},$ and so we have the following result.

\begin{proposition}
The generalized two-dimensional Boole type transformations (\ref{ET5.1a})
and (\ref{ET5.1b}) are ergodic with respect to the standard infinitesimal
measure $d\mu (x,y)=dxdy\ $ for $(x,y)\in \mathbb{R}^{2}.$ In particular,
the following equalities 
\begin{equation}
\int_{\mathbb{R}^{2}}f(\varphi _{2}(x,y))dxdy=\int_{\mathbb{R}%
^{2}}f(x,y)dxdy=\int_{\mathbb{R}^{2}}f(\psi _{\sigma _{(2)}\ }(x,y))dxdy
\label{ET5.7}
\end{equation}%
hold for any integrable function $f\in L_{1}(\mathbb{R}^{2};\mathbb{R}).$
\end{proposition}

The above result strongly suggests the validity of the following conjecture.

\begin{conjecture}
Let $\sigma \ \in \Sigma _{n}$ be any element (permutation) of the symmetric
group $\Sigma _{n},n\in {\mathbb{Z_{+}}}.$ Then the following generalized
Boole type transformation $\psi _{\sigma \ }:\mathbb{R}^{n}\backslash
\{0,0,...,0\}\rightarrow \mathbb{R}^{n},$ where 
\begin{equation}
\psi _{\sigma \ }(x_{1},x_{2},...,x_{n}):=(x_{1}-1/x_{\sigma (1)},x_{2}\pm
1/x_{\sigma (2)},...,x_{n}\pm 1/x_{\sigma (n)})\ \   \label{ET5.8}
\end{equation}%
is ergodic with respect to the standard infinitesimal measure $d\mu
(x_{1},x_{2},...,x_{n}):=$ $dx_{1}dx_{2}\cdots dx_{n},\ $\ $%
(x_{1},x_{2},...,x_{n})\in \mathbb{R}^{n}.$where   $n\in \mathbb{N},\ $%
permutations $\sigma $ $\in S_{n}$ and  signs $"\pm "$ are chosen from the
nondegeneracy \ \ \ condition $J_{\varphi }(x)\neq 0,x\in \mathbb{R}%
^{n}\backslash \{0\}.$
\end{conjecture}

For the case  $n=3,(x,y,z)$ $\in \mathbb{R}^{3}\backslash \{0,0,0\},$ one
obtains the following nontrivial three-dimensional Boole type mappings:%
\begin{eqnarray}
\varphi _{+}(x,y,z) &:&=(x-1/y,y+1/z,z+1/x),  \label{ET5.9} \\
\varphi _{-\ }(x,y,z) &:&=(x-1/y,y-1/z,z-1/x),\   \notag
\end{eqnarray}%
for which the Lebesgue measure $d\lambda (x,y,z)=dxdydz$ is also invariant
and eventually is ergodic. 

\section{Acknowledgements}

D.B. acknowledges the National Science Foundation (Grant CMMI-1029809), A.P.
and Y.P. acknowledge the Scientific and Technological Research Council of
Turkey (TUBITAK/NASU-111T558 Project) for a partial support of their
research.


\begin{thebibliography}{99}
\bibitem{Aa} Aaronson J. Ergodic theory for inner functions of the upper
half plane. Ann. Inst. H. Poincare, 1978, BXIV, p.233-253.

\bibitem{Aaa} Aaronson J. A remark on this existence of inner functions.
Journ. LMS, 1981, 23, p.469-474

\bibitem{AaJ} Aaronson J. \ The eigenvalues of nonsingular transformations.
Israel Journal of Math., 1983, 45, p.297-312

\bibitem{AaaJ} Aaronson J. An introduction to infinite ergodic theory. AMS,
v.50, 1997

\bibitem{AF} Ablowitz M.J. and Fokas A.S. Complex variables: introduction
and applications. Cambridge University Press, 1997

\bibitem{AW} Adler R. and Weiss B. The ergodic, infinite measure preserving
transformation of Boole. Israel Journal of Math., 1973, 16, p.263-278

\bibitem{BPP} Blackmore D. , Prykarpatsky A.K., Prykarpatsky Y.A.
Isospectral integrability analysis of dynamical systems on discrete
manifolds. Opuscula Mathematica, 32, 2012, No. 1, p. 39-54

\bibitem{BPS} Blackmore D., Prykarpatsky A.K., Samoylenko V.H. Nonlinear
dynamical systems of mathematical physics. NJ, World Scientific, 2011

\bibitem{BMS} Bogolubov N.N., Mitropolsky Yu.A. and Samoilenko A.M. \
Methods of accelerated convergence in nonlinear mechanics. Springer, 1976

\bibitem{Bo} Boole G. On the comparison of transcendents with certain
applications to the theory of definite integrals. Philos Transaction, Royal
Soc. London, 1857, v.147, p.745-803

\bibitem{Ha} Hardy G., Convergent series., Cambridge Press, 1947

\bibitem{HPP} Hentosh O.Ye., Prytula M.M. and Prykarpatsky A.K
Differential-geometric integrability fundamentals of nonlinear dynamical
systems on functional menifolds. (The second revised edition), Lviv
University Publisher, Lviv, Ukraine, 2006 (in Ukrainian)

\bibitem{KH} Katok A. and Hasselblatt Introduction to the Modern Theory of
dynamical systems. Cmbridge University Press, 1999

\bibitem{KN} Krein M.G. and Nudelman A.A. The Markov moment problem and
extremal tasks$.$ Moscow, Nauka Press, 1973 (in Russian)

\bibitem{PY} Pollycott M. and Yuri M. Dynamical systems and ergodic theory.
London Math. Society, Cambridge University Press, Student Texts, v.40, 1998

\bibitem{PS} Polya G., Sego H. Problems and solutions. Springer, NY, 1982

\bibitem{Pri} Privalov I.I. Boundary properties of analytical functions.
Gostekhizdatb Publ., Moscow, 1950

\bibitem{Pr} Prykarpatsky A.K. On invariant measure structure of a class of
ergodic discrete dynamical systems. Journal of Nonlinear Oscillations, 2000,
v.3, N1, p.78-83

\bibitem{PB} Prykarpatsky A.K. and Brzychczy S. On invariant measure
structure of a class of ergodic discrete dynamical systems. Proceedings of
the International Conference SCAN 2000/Interval 2000, September 19-22,
Karlsruhe, Germany

\bibitem{PF} Prykarpatsky A.K. and Feldman J. On the ergodic and special
properties of generalized Boole transformations. Proc. of Intern. Conference
"Difference equations, Special functions and orthogonal polynomials", held
25-30 July 2005 in Munich, Germany, p. 527-536

\bibitem{SUZ} Sagdeev R.Z., Usikov D.A., Zaslavsky G.M. Nonlinear physics:
from the pendulum to turbulence and chaos. Harwood Academic Publishers, 1988

\bibitem{Si} Sinai Ya.G. Ergodic theory. Nauka Publ., Moscow, 1984 (in
Russian)

\bibitem{Sk} Skorokhod A.V. Elements of the probability theory and causal
processes. Vyshcha Shkola Publ., Kyiv, 1975\ 

\bibitem{Ta} Takagi T. A simple example of the continuous function without
derivative. Proc. Phys.Math. Soc. Japan, 1903, 1, p. 176-177

\bibitem{WZ} Wheedon R., Zygmund A. Measure and integral: an introduction to
real analysis. Marcel Decker, Inc., NY, and \ Basel, 1977

\bibitem{YH} Yamaguti M and Hata M. Weierstrass's function and chaos.
Hokkaido Math. Journal, 1983, v.12, p. 333-342
\end{thebibliography}
\end{document}